\documentclass[11pt,a4paper,reqno,dvipsnames]{amsart}
\allowdisplaybreaks[4]

\usepackage{amsfonts}
\usepackage{amsmath}
\usepackage{amssymb}
\usepackage{amsthm}
\usepackage{amsxtra}
\usepackage{appendix}
\usepackage{bbm}
\usepackage{braket}
\usepackage{comment}
\usepackage{extarrows}
\usepackage{graphicx}
\usepackage{latexsym}
\usepackage{mathrsfs}
\usepackage{yhmath}
\usepackage{mathtools}

\usepackage{float}
\usepackage{booktabs}
\usepackage{verbatim}
\usepackage{xcolor}
\usepackage[normalem]{ulem}
\usepackage{caption,subcaption}
\usepackage{hyperref} 
\hypersetup{breaklinks, raiselinks}
\hypersetup{colorlinks, citecolor=blue, linkcolor=blue, urlcolor=blue}

\usepackage{bm}

\usepackage{tikz-cd}

\usepackage[initials,lite]{amsrefs} 
\AtBeginDocument{%
    \def\MR#1{}
}

\usepackage{cleveref}

\usepackage{fullpage}

\theoremstyle{plain}
\newtheorem{Theorem}{Theorem}[section]
\newtheorem{Lemma}[Theorem]{Lemma}
\newtheorem{Corollary}[Theorem]{Corollary}

\theoremstyle{definition}

\newtheorem{Assumptions and Discussion}[Theorem]{Assumptions and Discussion}

\newtheorem{Example}[Theorem]{Example}
\newtheorem{Definition}[Theorem]{Definition}

\theoremstyle{remark}

\newtheorem*{acknowledgment*}{Acknowledgment}

\usepackage[shortlabels]{enumitem}
\SetEnumerateShortLabel{a}{\textup{(\alph*)}}
\SetEnumerateShortLabel{A}{\textup{(\Alph*)}}
\SetEnumerateShortLabel{1}{\textup{(\arabic*)}}
\SetEnumerateShortLabel{i}{\textup{(\roman*)}}
\SetEnumerateShortLabel{I}{\textup{(\Roman*)}}

\begin{document}

\title{Projective dimension and regularity of  edge ideals of  some vertex-weighted oriented  unicyclic graphs}

\author{Guangjun Zhu$^{\ast}$ and  Hong Wang}

\date{\today }

\address{School of Mathematical Sciences, Soochow University, Suzhou, Jiangsu, 215006, P. R. China}

\email{zhuguangjun@suda.edu.cn(Corresponding author:Guangjun Zhu),\linebreak[4]
651634806@qq.com(Hong wang).}


\thanks{projective dimension, regularity,  edge ideal,   vertex-weighted oriented  unicyclic graph, vertex-weighted  rooted forest}

\maketitle
\begin{abstract}
In this paper we provide some exact formulas for the projective dimension and the regularity of edge ideals of  vertex-weighted oriented  unicyclic graphs.  These formulas are in function of the weight of the vertices, the numbers of edges. We  also give some examples to show that these  formulas are related to direction selection and the assumptions that  $w(x)\geq 2$ for any vertex $x$ cannot be dropped.
\end{abstract}

\section{Introduction}
Let $S=k[x_{1},\dots,x_{n}]$ be a polynomial ring in $n$ variables over a field $k$ and let $I\subset S$ be a homogeneous ideal.
There are two central invariants associated to $I$, the regularity $\mbox{reg}\,(I):=\mbox{max}\{j-i\ |\ \beta_{i,j}(I)\neq 0\}$
 and the projective dimension $\mbox{pd}\,(I):=\mbox{max}\{i\ |\ \beta_{i,j}(I)\neq 0\ \text{for some}\ j\}$, that in a sense, they measure the complexity of computing the graded Betti numbers $\beta_{i,j}(I)$ of $I$. In particular, if $I$ is a
monomial ideal,  its polarization  $I^{\mathcal{P}}$ has the same projective dimension and regularity as $I$ and is squarefree. Thus one can associate $I^{\mathcal{P}}$ to a graph or a hypergraph or  a simplicial complex.
  Many authors have studied the regularity
and Betti numbers of edge ideals of graphs, e.g. \cites{AF2,BHTN,BHO,HT3,J,KM,MV,Z1,Z2,Z3}.

A {\em directed graph} or {\em digraph} $D$ consists of a finite set $V(D)$ of vertices, together
with a collection $E(D)$ of ordered pairs of distinct points called edges or
arrows. A vertex-weighted directed graph is a triplet $D=(V(D), E(D),w)$, where  $w$ is a weight function $w: V(D)\rightarrow \mathbb{N}^{+}$, where $N^{+}=\{1,2,\ldots\}$.
Some times for short we denote the vertex set $V(D)$ and the edge set $E(D)$
by $V$ and $E$ respectively.
The weight of $x_i\in V$ is $w(x_i)$, denoted by $w_i$ or $w_{x_i}$.

The edge ideal of a vertex-weighted digraph was first introduced by Gimenez et al \cite{GBSVV}. Let $D=(V,E,w)$ be a  vertex-weighted digraph with the vertex set $V=\{x_{1},\ldots,x_{n}\}$. We consider the polynomial ring $S=k[x_{1},\dots, x_{n}]$ in $n$ variables over a field $k$. The edge ideal of $D$,  denoted by $I(D)$, is the ideal of $S$ given by
\[I(D)=(x_ix_j^{w_j}\mid  x_ix_j\in E).\]

Edge ideals of weighted digraphs arose in the theory of Reed-Muller codes as initial ideals of vanishing ideals
of projective spaces over finite fields \cites{MPV,PS}.
If a vertex $x_i$ of $D$ is a source (i.e., has only arrows leaving $x_i$) we shall always
assume $w_i=1$ because in this case the definition of $I(D)$ does not depend on the
weight of $x_i$.  If  $w_j=1$ for all $j$, then $I(D)$ is the edge ideal of underlying graph $G$ of $D$.
It has been extensively studied in the literature  \cites{HT3,MV,SVV}.
Especially the study of algebraic invariants corresponding to their minimal free resolutions has become popular(see
 \cites{AF2,BHTN,BHO,J,KM,Z1,Z2,Z3}). In \cite{Z3}, the first three authors  derive some exact formulas for the projective
dimension and regularity of the edge ideals associated to  some vertex-weighted digraphs such as rooted forests,  oriented cycles.
To the best of our knowledge, little is known about the projective
dimension  and the regularity of $I(D)$ for some  vertex-weighted digraphs.

In this article, we are interested in algebraic properties corresponding to the projective
dimension and the regularity of the edge ideals of  vertex-weighted oriented  unicyclic graphs.
 By using the approaches of Betti splitting and polarization, we derive some exact formulas  for the projective
dimension  and the regularity  of these edge ideals.
The results are as follows:

\begin{Theorem}
Let $T_j=(V(T_j),E(T_j),w_j)$ be a vertex-weighted  rooted forest such that $w(x)\geq 2$ if $d(x)\neq 1$ for  $1\leq j\leq s$ and $C_n$ a vertex-weighted oriented
cycle with  vertex set  $\{x_1,x_2,\ldots,x_n\}$ and $w_{x_i}\geq 2$ for   $1\leq i\leq n$. Let $D=C_n\cup(\bigcup\limits_{j=1}^{s}T_j)$ be a vertex-weighted oriented graph obtained by attaching  the root of $T_j$  to vertex $x_{i_j}$ of the cycle $C_{n}$ for $1\leq j\leq s$.
 Then
\begin{enumerate}[1]
\item $\mbox{reg}\,(I(D))=\sum\limits_{x\in V(D)}w(x)-|E(D)|+1$,
\item $\mbox{pd}\,(I(D))=|E(D)|-1$.
\end{enumerate}
\end{Theorem}

\begin{Theorem}
Let $D=(V(D),E(D),w)$ be a  vertex-weighted oriented unicyclic graph such that  $w(x)\geq 2$  if  $d(x)\neq 1$. Then
\begin{enumerate}[1]
\item $\mbox{reg}\,(I(D))=\sum\limits_{x\in V(D)}w(x)-|E(D)|+1$,
\item $\mbox{pd}\,(I(D))=|E(D)|-1$.
\end{enumerate}
\end{Theorem}

Our paper is organized as follows. In section $2$, we recall some
definitions and basic facts used in the following sections. In section $3$, we
provide some exact formulas for the projective
dimension and the regularity of  the edge
ideals of vertex-weighted oriented unicyclic graphs.
 Meanwhile, we give some examples to show the projective
dimension and the regularity  of these edge ideals are related to
direction selection and  the assumption  that
$w(x)\geq 2$  for any vertex $x$ cannot be dropped.

For all unexplained terminology and additional information, we refer to \cite{JG} (for the theory
of digraphs), \cite{BM} (for graph theory), and \cites{BH,HH} (for the theory of edge ideals of graphs and
monomial ideals). We greatfully acknowledge the use of the computer algebra system CoCoA (\cite{C}) for our experiments.

\section{Preliminaries}
In this section, we gather together the needed  definitions and basic facts, which will
be used throughout this paper. However, for more details, we refer the reader to \cites{AF2,BM,FHT,HT1,HH,J,JG,MPV,PRT,Z3}.

A {\em directed graph} or {\em digraph} $D$ consists of a finite set $V(D)$ of vertices, together
with a collection $E(D)$ of ordered pairs of distinct points called edges or
arrows. If $\{u,v\}\in E(D)$ is an edge, we write $uv$ for $\{u,v\}$, which is denoted to be the directed edge
where the direction is from $u$ to $v$ and $u$ (resp. $v$) is called the {\em starting}  point (resp. the {\em ending} point).
An {\em  oriented} graph is a  directed  graph  having no bidirected edges (i.e. each pair of vertices is joined by a single edge having a unique direction).
In other words an oriented graph $D$ is a simple graph $G$ together with an orientation
of its edges. We call $G$ the underlying graph of $D$.

Every concept that is valid for graphs automatically applies to digraphs too.
For example, the degree of a vertex $x$ in a digraph $D$, denoted by $d(x)$, is simply the degree of $x$ in
$G$. Likewise, a digraph is said to be connected if
its underlying graph is connected. A digraph $H$ is called an induced subgraph of digraph $D$  if  $V(H)\subseteq V(D)$,
and for any $x,y\in V(H)$,  $xy$ is an edge in $H$ if and only if  $xy$ is an edge in $D$.
For  $P\subset V(D)$, we  denote
$D\setminus P$ the induced subgraph of $D$ obtained by removing the vertices in $P$ and the
edges incident to these vertices. If $P=\{x\}$ consists of a single element, then we write $D\setminus x$ for $D\setminus \{x\}$.
For $W\subseteq E(D)$, we define $D\setminus W$ to be the subgraph of $D$ with all edges  in $W$  deleted (but its vertices remained).
When $W=\{e\}$ consists of a single edge, we write $D\setminus e$ instead of $D\setminus \{e\}$.
An {\em  oriented}  path or {\em  oriented}  cycle  is an orientation of a
path or cycle in which each vertex dominates its successor in the sequence.
An {\em  oriented} acyclic graph is a simple digraph without oriented cycles.
An {\em  oriented tree} or {\em polytree} is an  oriented acyclic graph formed by orienting the edges of undirected acyclic graphs.
A  {\em rooted tree} is an oriented tree in which all edges  are oriented  either away from or
towards the root.  Unless specifically stated, a rooted tree in this article
 is an oriented tree in which all edges  are oriented away from  the root.
An {\em oriented forest} is a disjoint union of oriented trees. A {\em rooted forest} is a disjoint union of rooted trees.

 A vertex-weighted oriented graph is a triplet $D=(V(D), E(D),w)$, where $V(D)$ is the  vertex set,
$E(D)$ is the edge set and $w$ is a weight function $w: V(D)\rightarrow \mathbb{N}^{+}$, where $N^{+}=\{1,2,\ldots\}$.
Some times for short we denote the vertex set $V(D)$ and edge set $E(D)$
by $V$ and $E$ respectively.
The weight of $x_i\in V$ is $w(x_i)$, denoted by $w_i$  or $w_{x_i}$.
Given  a  vertex-weighted oriented graph $D=(V,E,w)$ with the vertex set $V=\{x_{1},\ldots,x_{n}\}$, we  consider the polynomial
 ring $S=k[x_{1},\dots, x_{n}]$ in $n$ variables over a field $k$. The edge ideal of $D$,  denoted by $I(D)$, is the ideal of $S$ given by
\[I(D)=(x_ix_j^{w_j}\mid  x_ix_j\in E).\]

If a vertex $x_i$ of $D$ is a source (i.e., has only arrows leaving $x_i$) we shall always
assume $w_i=1$ because in this case the definition of $I(D)$ does not depend on the
weight of $x_i$.

For any homogeneous ideal $I$ of the polynomial ring  $S=k[x_{1},\dots,x_{n}]$, there exists a {\em graded
minimal finite free resolution}

\[0\rightarrow \bigoplus\limits_{j}S(-j)^{\beta_{p,j}(I)}\rightarrow \bigoplus\limits_{j}S(-j)^{\beta_{p-1,j}(I)}\rightarrow \cdots\rightarrow \bigoplus\limits_{j}S(-j)^{\beta_{0,j}(I)}\rightarrow I\rightarrow 0,\]
where the maps are exact, $p\leq n$, and $S(-j)$ is the $S$-module obtained by shifting
the degrees of $S$ by $j$. The number
$\beta_{i,j}(I)$, the $(i,j)$-th graded Betti number of $I$, is
an invariant of $I$ that equals the number of minimal generators of degree $j$ in the
$i$th syzygy module of $I$.
Of particular interests are the following invariants which measure the “size” of the minimal graded
free resolution of $I$.
The projective dimension of $I$, denoted pd\,$(I)$, is defined to be
\[\mbox{pd}\,(I):=\mbox{max}\,\{i\ |\ \beta_{i,j}(I)\neq 0\}.\]
The regularity of $I$, denoted $\mbox{reg}\,(I)$, is defined by
\[\mbox{reg}\,(I):=\mbox{max}\,\{j-i\ |\ \beta_{i,j}(I)\neq 0\}.\]

We now derive some formulas for $\mbox{pd}\,(I)$  and $\mbox{reg}\,(I)$ in some special cases by using some
tools developed in \cite{FHT}.

\begin{Definition} \label{bettispliting}Let $I$  be a monomial ideal, and suppose that there exist  monomial
ideals $J$ and $K$ such that $\mathcal{G}(I)$ is the disjoint union of $\mathcal{G}(J)$ and $\mathcal{G}(K)$, where $\mathcal{G}(I)$ denotes the unique minimal set of monomial generators of $I$. Then $I=J+K$
is a {\em Betti splitting} if
\[\beta_{i,j}(I)=\beta_{i,j}(J)+\beta_{i,j}(K)+\beta_{i-1,j}(J\cap K)\hspace{2mm}\mbox{for all}\hspace{2mm}i,j\geq 0,\]
where $\beta_{i-1,j}(J\cap K)=0\hspace{2mm}  \mbox{if}\hspace{2mm} i=0$.
\end{Definition}

In  \cite{FHT}, the authors describe some sufficient conditions for an
ideal $I$ to have a Betti splitting. We need  the following lemma.

\begin{Lemma}\label{lem:1}{\em (\cite[Corollary 2.7]{FHT})}
Suppose that $I=J+K$ where $\mathcal{G}(J)$ contains all
the generators of $I$ divisible by some variable $x_{i}$ and $\mathcal{G}(K)$ is a nonempty set containing
the remaining generators of $I$. If $J$ has a linear resolution, then $I=J+K$ is a Betti
splitting.
\end{Lemma}

When $I$ is a Betti splitting ideal, Definition \ref{bettispliting} implies the following results:
\begin{Corollary} \label{cor:1}
If $I=J+K$ is a Betti splitting ideal, then
\begin{enumerate}[1]
 \item $\mbox{reg}\,(I)=\mbox{max}\{\mbox{reg}\,(J),\mbox{reg}\,(K),\mbox{reg}\,(J\cap K)-1\}$,
 \item $\mbox{pd}\,(I)=\mbox{max}\{\mbox{pd}\,(J),\mbox{pd}\,(K),\mbox{pd}\,(J\cap K)+1\}$.
\end{enumerate}
\end{Corollary}

The following lemmas is often used in this article.
\begin{Lemma}
\label{lem:2}{\em (\cite[Lemma 1.3]{HTT}) }
Let $R$ be a polynomial ring over a field and let $I$ be a proper non-zero homogeneous
ideal in $R$. Then
\begin{enumerate}[1]
\item $\mbox{pd}\,(I)=\mbox{pd}\,(S/I)-1$,
\item $\mbox{reg}\, (I)=\mbox{reg}\,(S/I)+1$.
\end{enumerate}
\end{Lemma}

\begin{Lemma}
\label{lem:3}{\em (\cite[Lemma 2.2 and  Lemma 3.2 ]{HT1})}
Let $S_{1}=k[x_{1},\dots,x_{m}]$, $S_{2}=k[x_{m+1},\dots,x_{n}]$ and $S=k[x_{1},\dots,x_{n}]$ be three polynomial rings, $I\subseteq S_{1}$ and
$J\subseteq S_{2}$ be two proper non-zero homogeneous  ideals.  Then
\begin{enumerate}[1]
\item $\mbox{pd}\,(S/(I+J))=\mbox{pd}\,(S_{1}/I)+\mbox{pd}\,(S_{2}/J)$,
\item $\mbox{reg}\,(S/(I+J))=\mbox{reg}\,(S_{1}/I)+\mbox{reg}\,(S_{2}/J)$.
\end{enumerate}
\end{Lemma}

From Lemma \ref{lem:2} and Lemma \ref{lem:3}, we have
\begin{Lemma}
\label{lem:4}{\em (\cite[Lemma 3.1]{Z1})}
Let $S_{1}\!\!=\!k[x_{1},\dots,x_{m}]$ and $S_{2}\!=\!k[x_{m+1},\dots,x_{n}]$ be two polynomial rings, $I\subseteq S_{1}$ and
$J\subseteq S_{2}$ be two non-zero homogeneous  ideals. Then
\begin{enumerate}[1]
\item $\mbox{pd}\,(I+J)=\mbox{pd}\,(I)+\mbox{pd}\,(J)+1$,
\item $\mbox{reg}\,(I+J)=\mbox{reg}\,(I)+\mbox{reg}\,(J)-1$.
\end{enumerate}
\end{Lemma}

  Let $\mathcal{G}(I)$ denote the minimal set of generators of a monomial ideal $I\subset S$
 and let $u\in S$ be a monomial, we set $\mbox{supp}(u)=\{x_i: x_i|u\}$. If $\mathcal{G}(I)=\{u_1,\ldots,u_m\}$, we set $\mbox{supp}(I)=\bigcup\limits_{i=1}^{m}\mbox{supp}(u_i)$. The following lemma is well known.
 \begin{Lemma}\label{lem:5}
  Let $I,J\!=(u)$ be two monomial ideals  such that $\mbox{supp}\,(u)\cap \mbox{supp}\,(I)\!\!=\emptyset$. If the degree of monomial $u$  is  $d$. Then
\begin{enumerate}[1]
\item $\mbox{reg}\,(J)=d$,
\item $\mbox{reg}\,(JI)=\mbox{reg}\,(I)+d$,
\item $\mbox{pd}\,(JI)=\mbox{pd}\,(I)$.
\end{enumerate}
\end{Lemma}

\begin{Definition} \label{polarization}
Suppose that $u=x_1^{a_1}\cdots x_n^{a_n}$ is a monomial in $S$. We define the {\it polarization} of $u$ to be
the squarefree monomial
\[\mathcal{P}(u)=x_{11}x_{12}\cdots x_{1a_1} x_{21}\cdots x_{2a_2}\cdots x_{n1}\cdots x_{na_n}\]
in the polynomial ring $S^{\mathcal{P}}=k[x_{ij}\mid 1\leq i\leq n, 1\leq j\leq a_i]$.
If $I\subset S$ is a monomial ideal with $\mathcal{G}(I)=\{u_1,\ldots,u_m\}$,  the  {\it polarization}
of $I$,  denoted by $I^{\mathcal{P}}$, is defined as:
\[I^{\mathcal{P}}=(\mathcal{P}(u_1),\ldots,\mathcal{P}(u_m)),\]
which is a squarefree monomial ideal in the polynomial ring $S^{\mathcal{P}}$.
\end{Definition}

Here is an example of how polarization works.

\begin{Example} \label{exm:1} Let $I(D)=(x_1x_2^{3},x_2x_3^2,x_3x_4^{4},x_4x_1^{5})$ be the edge ideal of a vertex-weighted  oriented cycle $D$, then the polarization of $I(D)$ is the ideal  $I(D)^{\mathcal{P}}=(x_{11}x_{21}x_{22}x_{23},x_{21}x_{31}x_{32},\\
x_{31}x_{41}x_{42}x_{43}x_{44},x_{41}x_{11}x_{12}x_{13}x_{14}x_{15})$.
\end{Example}

A monomial ideal $I$ and its polarization $I^{\mathcal{P}}$ share many homological and
algebraic properties.  The following is a very useful property of polarization.

\begin{Lemma}\label{lem:6}{\em (\cite[Corollary 1.6.3]{HH})}
Let $I\subset S$ be a monomial ideal and $I^{\mathcal{P}}\subset S^{\mathcal{P}}$ its polarization.
Then
\begin{enumerate}[1]
\item $\beta_{ij}(I)=\beta_{ij}(I^{\mathcal{P}})$ for all $i$ and $j$,
\item $\mbox{reg}\,(I)=\mbox{reg}\,(I^{\mathcal{P}})$,
\item $\mbox{pd}\,(I)=\mbox{pd}\,(I^{\mathcal{P}})$.
\end{enumerate}
\end{Lemma}

The following lemma can be used for computing the projective dimension and the
regularity of an ideal.
\begin{Lemma}
\label{lem:7}{\em (\cite[Lemma 1.1 and Lemma 1.2]{HTT})}  Let\ \ $0\rightarrow A \rightarrow  B \rightarrow  C \rightarrow 0$\ \  be a short exact sequence of finitely generated graded $S$-modules.
Then
\begin{enumerate}[1]
\item $\mbox{reg}\,(B)\leq \mbox{max}\,\{\mbox{reg}\,(A), \mbox{reg}\,(C)\}$,
\item $\mbox{pd}\,(B)\leq \mbox{max}\,\{\mbox{pd}\,(A), \mbox{pd}\,(C)\}$.
\end{enumerate}
\end{Lemma}

\section{Projective dimension and regularity of  edge ideals of  vertex-weighted oriented  unicyclic graphs}
In this section, we will provide some exact formulas for the projective dimension and the regularity of  the edge ideals of vertex-weighted  unicyclic graphs. Meanwhile, we give some examples to show the projective dimension and the regularity of  the edge
ideals of vertex-weighted oriented unicyclic graphs are related to direction selection and
the assumption that $w(x)\geq 2$  if $d(x)\neq 1$  cannot be dropped.
We shall start from the definition of the union of some digraphs.

\begin{Definition} \label{def:2} Let $D_{i}=(V(D_{i}),E(D_{i}))$  be a digraph with the underlying graph $G_i$
for  $1\leq i\leq k$. If  $e$ is an edge of  $G_j$ for $i_1\leq j\leq i_s$ with $1\leq i_1\leq \cdots\leq i_s\leq k$, then
the direction of edge $e$ is the same in  $D_j$ for $i_1\leq j\leq i_s$. The union of digraphs $D_{1},D_{2},\dots, D_{k}$,  written  as $\bigcup\limits_{i=1}^{k}D_{i}$, is the digraph
with vertex set $\bigcup\limits_{i=1}^{k} V(D_{i})$ and edge set $\bigcup\limits_{i=1}^{k} E(D_{i})$.
\end{Definition}

The following two lemmas is needed to facilitate calculating the projective dimension
and the regularity of the edge ideal of some graphs.

\begin{Lemma}\label{lem:8}{\em (\cite[Theorem 3.5, Theorem 3.3 ]{Z3})}
Let $D=(V(D),E(D),w)$ be a vertex-weighted rooted forest such that   $w(x)\geq 2$  if $d(x)\neq 1$. Then
\begin{enumerate}[1]
\item $\mbox{reg}\,(I(D))=\sum\limits_{x\in V(D)}w(x)-|E(D)|+1$,
\item $\mbox{pd}\,(I(D))=|E(D)|-1$.
\end{enumerate}
\end{Lemma}

\begin{Lemma}\label{lem:9}{\em (\cite[Theorem 4.1]{Z3})}
Let $D=(V(D),E(D),w)$ be  a vertex-weighted oriented cycle such that   $w(x)\geq 2$  for any $x\in V(D)$. Then
\begin{enumerate}[1]
\item $\mbox{reg}\,(I(D))=\sum\limits_{x\in V(D)}w(x)-|E(D)|+1$,
\item $\mbox{pd}\,(I(D))=|E(D)|-1$.
\end{enumerate}
\end{Lemma}

Let $D=(V,E)$ be  a digraph and  $x\in V$, then we call  $N_D^{+}(x)=\{y: xy \in E\}$ and $N_D^{-}(x)=\{y: yx \in E\}$ to be  the out-neighbourhood and  in-neighbourhood of $x$, respectively. The neighbourhood of $x$ is the set $N_D(x)=N_D^{+}(x)\cup N_D^{-}(x)$.

Now we are ready to present the main result of this section.
\begin{Theorem}\label{thm:3}
Let $T_j=(V(T_j),E(T_j),w_j)$ be a vertex-weighted  rooted forest such that $w(x)\geq 2$ if $d(x)\neq 1$ for  $1\leq j\leq s$ and $C_n$ a vertex-weighted oriented
cycle with  vertex set  $\{x_1,x_2,\ldots,x_n\}$ and $w_{x_i}\geq 2$ for   $1\leq i\leq n$. Let $D=C_n\cup(\bigcup\limits_{j=1}^{s}T_j)$ be a vertex-weighted oriented graph obtained by attaching  the root of $T_j$  to vertex $x_{i_j}$ of the cycle $C_{n}$ for $1\leq j\leq s$.
 Then
\begin{enumerate}[1]
\item $\mbox{reg}\,(I(D))=\sum\limits_{x\in V(D)}w(x)-|E(D)|+1$,
\item $\mbox{pd}\,(I(D))=|E(D)|-1$.
\end{enumerate}
\end{Theorem}
\begin{proof} Let $N_{T_{j}}^{+}(x_{i_j})\!=\!\{y_{j1},\ldots,y_{j,t_j}\}$ for $1\leq j\leq s$.  Assume that $i_1\!<\!i_2\!<\!\cdots\!<\!i_s$. We consider two cases:

 Case (1): If  $T_j$ is a rooted star graph for any $1\leq j\leq s$, then we have
\begin{align*}
I(D) =& (x_1x_2^{w_2},\ldots,x_{n-1}x_n^{w_n},x_{n}x_1^{w_1}, x_{i_1}y_{11}^{w_{y_{11}}},\ldots,x_{i_1}y_{1,t_1}^{w_{y_{1,t_1}}},x_{i_2}y_{21}^{w_{y_{21,}}},\ldots,  x_{i_2}y_{2,t_2}^{w_{y_{2,t_2}}},\ldots,\\
&x_{i_s}y_{s1}^{w_{y_{s1}}},\ldots,x_{i_s}y_{s,t_s}^{w_{y_{s,t_s}}}).
\end{align*}
Case (2):  If there exists some $1\leq j\leq s$ such that $T_j$ is a  rooted  tree but not a star  digraph. We may assume that
 $T_j$ is a  rooted  tree but not a star  digraph for any $1\leq j\leq s$. Because other cases follow the same line of arguments.  In this case, we have
\begin{align*}
I(D)=&(x_1x_2^{w_2}\!,\ldots,x_{n-1}x_n^{w_n}\!,x_{n}x_1^{w_1}\!, x_{i_1}y_{11}^{w_{y_{11}}}\!\!\!,\ldots,x_{i_1}y_{1,t_1}^{w_{y_{1,t_1}}}\!\!\!\!\!, x_{i_2}y_{21}^{w_{y_{21}}}\!\!\!,\ldots,x_{i_2}y_{2,t_2}^{w_{y_{2,t_2}}}\!\!\!\!, \\
& \ldots, x_{i_s}y_{s1}^{w_{y_{s1}}}\!\!\!,\ldots,x_{i_s}y_{s,t_s}^{w_{y_{s,t_s}}})+
I(T_{1}\setminus x_{i_1})+I(T_{2}\setminus x_{i_2})+\cdots+I(T_{s}\setminus x_{i_s}),
\end{align*}
where $I(T_{j}\setminus x_{i_j})$ is the edge ideal of the  induced subgraph of $T_{j}$ obtained by removing  vertex $x_{i_j}$ and the
edges incident to $x_{i_j}$ for $j=1,\ldots,s$.

Case (1) can be shown by similar arguments as Case (2), so we only prove the statement holds in Case (2).

Let $I(D)^{\mathcal{P}}$ be the polarization of $I(D)$, then
\begin{align*}
I(D)^{\mathcal {P}} = &(x_{11}\!\prod\limits_{j=1}^{w_{2}}x_{2j},\ldots,
x_{n-1,1}\prod\limits_{j=1}^{w_{n}}x_{nj},x_{n1}\prod\limits_{j=1}^{w_{1}}x_{1j},
x_{i_11}\prod\limits_{j=1}^{w_{y_{11}}}y_{11,\,j},\ldots, x_{i_11}\!\!\!\prod\limits_{j=1}^{w_{y_{1,t_1}}}\!\!\!y_{1,t_1,\,j},x_{i_21}\!\!\prod\limits_{j=1}^{w_{y_{21}}}\!\!y_{21,\,j},\\
 &\ldots,x_{i_21}\!\!\prod\limits_{j=1}^{w_{y_{2,t_2}}}\!\!y_{2,t_2,\,j},\ldots,
x_{i_s1}\prod\limits_{j=1}^{w_{y_{s1}}}y_{s1j},\ldots, x_{i_s1}\!\prod\limits_{j=1}^{w_{y_{s,t_s}}}y_{s,t_s,\,j})+I(T_{1}\setminus x_{i_1})^{\mathcal{P}}+I(T_{2}\setminus x_{i_2})^{\mathcal{P}}\\
+&\cdots+I(T_{s}\setminus x_{i_s})^{\mathcal{P}}.
\end{align*}
We set $i_1=1$,  $K=(x_{n1}\!\prod\limits_{j=1}^{w_{1}}x_{1j})$ and
\begin{align*}
J = &(x_{11}\!\!\prod\limits_{j=1}^{w_{2}}x_{2j},\ldots,x_{n-1,1}\!\!\prod\limits_{j=1}^{w_{n}}x_{nj},
\widehat{x_{n1}\!\!\prod\limits_{j=1}^{w_{1}}x_{1j}}, x_{11}\!\!\prod\limits_{j=1}^{w_{y_{11}}}y_{11,\,j},\ldots,\linebreak x_{11}\!\!\prod\limits_{j=1}^{w_{y_{1,t_1}}}y_{1,t_1,\,j}, x_{i_21}\prod\limits_{j=1}^{w_{y_{21}}}y_{21,\,j},\ldots,\\
   &x_{i_21}\prod\limits_{j=1}^{w_{y_{2,t_2}}}\!\!y_{2,t_2,\,j},\ldots,
x_{i_s1}\prod\limits_{j=1}^{w_{y_{s1}}}y_{s1,\,j},\ldots,x_{i_s1}\prod\limits_{j=1}^{w_{y_{s,t_s}}}y_{s,t_s,\,j}) +I(T_{1}\setminus x_{i_1})^{\mathcal{P}}+I(T_{2}\setminus x_{i_2})^{\mathcal{P}}\\
 +&\cdots+I(T_{s}\setminus x_{i_s})^{\mathcal{P}}
\end{align*}
where $\widehat{x_{n1}\!\prod\limits_{j=1}^{w_{1}}x_{1j}}$ denotes the element $x_{n1}\!\prod\limits_{j=1}^{w_{1}}x_{1j}$ being omitted from the ideal $J$.

Note that $J$ is  actually the polarization of the edge ideal $I(D\setminus e)$ of the  subgraph $D\setminus e$ where $e=x_nx_1$,
and $D\setminus e$ is a  rooted tree, whose  root is $x_1$. By Lemmas \ref{lem:6} and \ref{lem:8}, we obtain
\begin{align*}
\mbox{reg}\,(J)&=\mbox{reg}\,(J^{\mathcal {P}})=\sum\limits_{x\in V(D\setminus e)}w(x)-|E(D\setminus e)|+1\\
 &= \sum\limits_{x\in V(D)}w(x)-(w_1-1)-(|E(D)|-1)+1\\
 &= \sum\limits_{x\in V(D)}w(x)-|E(D)|+1+2-w_1\\
 &= \sum\limits_{x\in V(D)}w(x)-|E(D)|+3-w_1
\end{align*}
where third equality holds because we have weighted  one in  vertex $x_1$ in  the  expression $\sum\limits_{x\in V(D\setminus e)}w(x)$,
and
\[\mbox{pd}\,(J)=\mbox{pd}\,(J^{\mathcal {P}})=|E(D\setminus e)|-1=|E(D)|-2.\]

Now, we will compute $\mbox{reg}\,(J\cap K)$ and $\mbox{pd}\,(J\cap K)$.  We distinguish into the following two cases:

(1) If $s=1$, or  $s\geq 2$ and $i_s\neq n$, then $d(x_n)=2$. In this case,  we set $J\cap K=KL$. We  write $L$ as follows:
\[
L=L_1+L_2
\]
  where $L_1=(\prod\limits_{j=1}^{w_{2}}x_{2j},x_{n-1,1}\!\prod\limits_{j=2}^{w_{n}}x_{nj},  \prod\limits_{j=1}^{w_{y_{11}}}y_{11,j},\ldots,\prod\limits_{j=1}^{w_{y_{1k}}}y_{1k,j})+I(D\setminus \{x_1,x_n\})^{\mathcal{P}}$
is the polarization of the edge ideal of a rooted forest $H$,
here $H$  is the union of the induced subgraph  $D\setminus \{x_1,x_n\}$ of $D$ and a vertex-weighted oriented  graph $H'$
with the vertex set $V(H')=\{x_{21},x_{22},x_{n-1,1},x_{n2},y_{11,1},y_{11,2},\ldots,y_{1,k,1},y_{1,k,2}\}$, the  edge  set
 $E(H')=\{x_{21}x_{22},x_{n-1,1}x_{n2},\\
 y_{11,1}y_{11,2},\ldots,y_{1,k,1}y_{1,k,2}\}$
 and  a weight function $w'\!:\!V(H')\rightarrow \mathbb{N}^{+}$ such that
$w'(x_{21})=1$, $w'(x_{22})=w_{2}-1$, $w'(x_{n-1,1})=1$, $w'(x_{n2})=w_{n}-1$,
 $w'(y_{1,j,1})=1$,  $w'(y_{1,j,2})=w_{y_{1,j}}-1\geq 1$ for any $1\leq j\leq k$,
 and $L_2=(\prod\limits_{j=1}^{w_{y_{1,k+1}}}\!\!y_{1,k+1,\,j},
\ldots,\!\!\!\prod\limits_{j=1}^{w_{y_{1,t_1}}}\!\!y_{1,t_1,\,j})$ for any $k+1\leq \ell\leq t_1$.
Thus  $|E(H)|=|E(D)|-(t_1+3)+(k+2)=|E(D)|-(t_1-k+1)$.  We consider the following two cases:

$(i)$ If $k<t_1$,  notice that the variables that appear in  $K$ and  $L$, and  in  $L_1$
and  $L_2$ are different, respectively. Then using Lemmas \ref{lem:4},  \ref{lem:5} and \ref{lem:8}, we obtain
\begin{align*}
\mbox{reg}\,(J\cap K) &= (1+w_1)+\mbox{reg}\,(L)=(1+w_1)+\mbox{reg}\,(L_{1})+\mbox{reg}\,(L_{2})-1\\
 &= (1+w_1)+(\sum\limits_{x\in V(H)}w(x)-|E(H)|+1)\\
 &+ ((\sum\limits_{j=k+1}^{t_1}w_{y_{1j}}-(t_1-k))+1)-1\\
 &= (1+w_1+\sum\limits_{x\in V(H)}w(x)+\sum\limits_{j=k+1}^{t_1}w_{y_{1j}})\\
 &- (|E(D)|-(t_1-k+1))-(t_1-k)+1\\
 &= \sum\limits_{x\in V(D)}w(x)-|E(D)|+2
\end{align*}
where  the last equality holds because of  $x_{n2}\in V(H')$ with $w'(x_{n2})=w_{n}-1$ and $x_1\notin V(H)$, and
\begin{align*}
\mbox{pd}\,(J\cap K) &= \mbox{pd}\,(L)=\mbox{pd}\,(L_{1}+L_{2})=\mbox{pd}\,(L_{1})+\mbox{pd}\,(L_{2})+1\\
 &= |E(H)|-1+(t_1-k-1)+1\\
 &= (|E(D)|-(t_1-k+1))-1+(t_1-k-1)+1\\
 &= |E(D)|-2.
\end{align*}

$(ii)$  If $k=t_1$, then $L_2=(0)$  and $|E(H)|=|E(D)|-(t_1+3)+(t_1+2)=|E(D)|-1$. Again  by Lemmas \ref{lem:4}, \ref{lem:5}  and \ref{lem:8}, we obtain
\begin{align*}
\mbox{reg}\,(J\cap K) &= \mbox{reg}\,(KL)=(1+w_1)+\mbox{reg}\,(L)\\
 &= (1+w_1)+\sum\limits_{x\in V(H)}w(x)-|E(H)|+1\\
 &= (1+w_1+\sum\limits_{x\in V(H)}w(x))-(|E(D)|-1)+1\\
 &= \sum\limits_{x\in V(D)}w(x)-|E(D)|+2
\end{align*}
where the last equality holds because of  $x_{n2}\in V(H')$ with $w'(x_{n2})=w_{n}-1$ and $x_1\notin V(H)$, and
\[
\mbox{pd}\,(J\cap K)=\mbox{pd}\,(KL)=\mbox{pd}\,(L)=|E(H)|-1=|E(D)|-2.
\]
In brief,  no matter  case $(i)$ or $(ii)$, we always have
\begin{align*}
    \mbox{reg}\,(J\cap K) &= \sum\limits_{x\in V(D)}w(x)-|E(D)|+2; \cr
   \mbox{pd}\,(J\cap K) &= |E(D)|-2.
  \end{align*}
Since $K$ has a linear resolution, it follows that $I(D)^{\mathcal{P}}=J+K$ is a Betti splitting. Thus by  Lemma \ref{lem:6} and Corollary \ref{cor:1},
we get
\begin{align*}
\mbox{reg}\,(I(D)) &= \mbox{reg}\,(I(D)^{\mathcal{P}})=\mbox{max}\{\mbox{reg}\,(J),\mbox{reg}\,(K),\mbox{reg}\,(J\cap K)-1\}\\
 &= \mbox{max}\{\!\!\!\!\sum\limits_{x\in V(D)}\!\!\!\!\!w(x)\!-|E(D)|+3\!-w_1,w_1+1,\!\!\!\!\!\sum\limits_{x\in V(D)}\!\!\!\!w(x)\!-|E(D)|+1\}\\
 &= \sum\limits_{x\in V(D)}w(x)-|E(D)|+1,
  \end{align*}
and
\begin{align*}
\mbox{pd}\,(I(D)) &= \mbox{pd}\,(I(D)^{\mathcal{P}})=\mbox{max}\{\mbox{pd}\,(J), \mbox{pd}\,(K), \mbox{pd}\,(J\cap K)+1\}\\
 &= \mbox{max}\{|E(D)|-2,0,|E(D)|-2+1\}\\
 &= |E(D)|-1.
  \end{align*}

(2) If $i_s=n$,  then $d(x_n)>2$.  In this case,  we still set $J\cap K=KL$ and  write $L$ as follows:
\[
L=L_1+L_2
\]
 where $L_1=(\prod\limits_{j=1}^{w_{2}}x_{2j},x_{n-1,1}\!\prod\limits_{j=2}^{w_{n}}x_{nj},  \prod\limits_{j=1}^{w_{y_{11}}}y_{11,\,j},\ldots,\prod\limits_{j=1}^{w_{y_{1k}}}y_{1k,\,j},
 \prod\limits_{j=1}^{w_{y_{s1}}}y_{s1,\,j},\ldots, \prod\limits_{j=1}^{w_{y_{s,\ell}}}y_{s,\ell,\,j})+I(D\setminus\{x_1,x_n\})^{\mathcal{P}}$
 is the polarization of the edge ideal of the graph $H$,  here $H$ is the union of the induced subgraph $D\setminus\{x_1,x_n\}$ of $D$ and the vertex-weighted oriented graph $H'$ with the vertex set $V(H')=\{x_{21},x_{22},x_{n-1},x_{n2},y_{11,1},y_{11,2},\ldots,y_{1k,1},y_{1k,2},
 y_{s1,1}, y_{s1,2},\ldots,y_{s,\ell,1},y_{s,\ell,2}\}$, the edge set
 $E(H')=\{x_{21}x_{22},x_{n-1}x_{n2},y_{11,1}y_{11,2},\ldots,
y_{1k,1}y_{1k,2}, y_{s1,1}y_{s1,2},\ldots, y_{s,\ell,1}y_{s,\ell,2}\}$ and a weight function $w': V(H')\rightarrow \mathbb{N}^{+}$ such that
$w'(x_{21})=w'(x_{n-1,1})=w'(y_{11,1})=\cdots=w'(y_{1k,1})=w'(y_{s1,1})=\cdots=w'(y_{s,\ell,1})=1$, $w'(x_{22})=w_{2}-1,w'(x_{n2})=w_{n}-1,w'(y_{1j,\,2})=w_{y_{1j}}-1\geq 1$ for any $1\leq j\leq k\leq t_1$
and $w'(y_{s,p,2})=w_{y_{s,p}}-1\geq 1$ for any $1\leq p\leq \ell \leq t_s$, and
 $L_2=(\prod\limits_{j=1}^{w_{y_{1,k+1}}}\!\!\!y_{1,k+1,\,j},\ldots,
 \!\prod\limits_{j=1}^{w_{y_{1,t_1}}}\!\!\!y_{1,t_1,\,j}, \prod\limits_{j=1}^{w_{y_{s,\ell+1}}}y_{s,\ell+1,\,j},
\ldots,\!\prod\limits_{j=1}^{w_{y_{s,t_s}}}y_{s,t_s,\,j})$.
In this case, $H$ is a forest and $|E(H)|=|E(D)|-(t_{s}-\ell)-(t_1-k)-1=|E(D)|-(t_{s}-\ell+t_1-k+1)$. We consider the following two cases:

$(i)$ If $k<t_1$ or $\ell<t_s$,  notice again that the variables that appear in  $K$ and  $L$, and  in  $L_1$ and  $L_2$ are different, respectively.  Then using Lemmas \ref{lem:4},  \ref{lem:5}  and \ref{lem:8}, we obtain
\begin{align*}
\mbox{reg}\,(J\cap K) &= (1+w_1)+\mbox{reg}\,(L)=(1+w_1)+\mbox{reg}\,(L_{1})+\mbox{reg}\,(L_{2})-1\\
 &= (1+\!w_1)+(\sum\limits_{x\in V(H)}w(x)-|E(H)|+1)\\
 &+ ((\sum\limits_{j=k+1}^{t_1}w_{y_{1j}}-(t_1-k))+1)-1\\
 &= (1+w_1+\sum\limits_{x\in V(H)}w(x)+\sum\limits_{j=k+1}^{t_1}w_{y_{1j}})-(|E(D)|\\
 &- (t_1-k+1))-(t_1-k)+1\\
 &= \sum\limits_{x\in V(D)}w(x)-|E(D)|+2
  \end{align*}
where the last equality holds because of  $x_{n2}\in V(H')$ with $w'(x_{n2})=w_{n}-1$ and $x_1\notin V(H)$, and
\begin{align*}
\mbox{pd}\,(J\cap K)&= \mbox{pd}\,(L)=\mbox{pd}\,(L_{1}+L_{2})=\mbox{pd}\,(L_{1})+\mbox{pd}\,(L_{2})+1\\
&= |E(H)|\!-\!1+\![\sum\limits_{j=k+1}^{t_1}\!\mbox{pd}\,((y_{1j}^{w_{y_{1j}}}))\!+\!\sum\limits_{j=\ell+1}^{t_s}\!
\mbox{pd}\,((y_{sj}^{w_{y_{sj}}}))\!+\!1]+\!1\\
&= [|E(D)|-(t_{s}-\ell+t_1-k+1)]-1+(t_s-\ell+t_1-k-1)+1\\
&= |E(D)|-2.
  \end{align*}
$(ii)$ If $k=t_1$ and $\ell=t_s$, then $L_2=(0)$ and $|E(H)|=|E(D)|-1$.
From (ii) of case (1), we also have
\begin{align*}
\mbox{reg}\,(J\cap K)&= \sum\limits_{x\in V(D)}w(x)-|E(D)|+2,\\
\mbox{pd}\,(J\cap K)&= |E(D)|-2.
  \end{align*}
Similar arguments as case (1), we obtain
\begin{align*}
\mbox{reg}\,(I(D))&= \sum\limits_{x\in V(D)}w(x)-|E(D)|+1,\\
\mbox{pd}\,(I(D))&= |E(D)|-1.
  \end{align*}
The proof is completed.
\end{proof}

\begin{Theorem}\label{thm:4}
Let $D=(V(D),E(D),w)$ be a  vertex-weighted oriented unicyclic graph such that  $w(x)\geq 2$  if  $d(x)\neq 1$. Then
\begin{enumerate}[1]
\item $\mbox{reg}\,(I(D))=\sum\limits_{x\in V(D)}w(x)-|E(D)|+1$,
\item $\mbox{pd}\,(I(D))=|E(D)|-1$.
\end{enumerate}
\end{Theorem}
\begin{proof} Let $D_1,\ldots,D_m$ be all  connected components of $D$.  Thus  by Lemma \ref{lem:4}, we get
\begin{align*}
\mbox{reg}\,(I(D))&= \mbox{reg}\,(\sum\limits_{i=1}^{m}I(D_i))=\sum\limits_{i=1}^{m}\mbox{reg}\,(I(D_i))-(m-1),\\
\mbox{pd}\,(I(D))&= \mbox{pd}\,(\sum\limits_{i=1}^{m}I(D_i))=\sum\limits_{i=1}^{m}\mbox{pd}\,(I(D_i))+(m-1).
  \end{align*}

By Lemma \ref{lem:8},  it is enough to show
\begin{align*}
\mbox{reg}\,(I(D))&= \sum\limits_{x\in V(D)}w(x)-|E(D)|+1,\\
\mbox{pd}\,(I(D))&= |E(D)|-1
  \end{align*}
 where $D$ is a connected vertex-weighted oriented unicyclic graph.
The conclusion follows from Theorem \ref{thm:3}.
\end{proof}

An immediate consequence of the above theorem is the following corollary.
\begin{Corollary}\label{cor:3}
Let $D=(V(D),E(D),w)$ be a vertex-weighted oriented unicyclic graph with $mm$ connected components. Let  $w(x)\geq 2$  for any $d(x)\neq 1$. Then
\[\mbox{depth}\,(I(D))=m.\]
\end{Corollary}
\begin{proof}
It follows from Auslander-Buchsbaum formula.
\end{proof}

The following two examples show that the assumption  that  $w(x)\geq 2$  if $d(x)\neq 1$ in Theorem \ref{thm:3} and Theorem \ref{thm:4}  cannot be dropped.
\begin{Example}  \label{exm:3}
Let $I(D)=(x_1x_2^{2},x_2x_3^2,x_3x_4^{2},x_4x_5^2,x_5x_1^{2},x_1x_6,x_6x_7,x_7x_8^{2})$ be the edge ideal of a vertex-weighted oriented unicyclic  graph $D=(V(D),E(D),w)$  with $w_1=w_2=w_3=w_4=w_5=w_8=2$ and
$w_6=w_7=1$. By using CoCoA, we get $\mbox{reg}\,(I(D))=8$ and $\mbox{pd}\,(I(D))=6$.  But  we have
$\mbox{reg}\,(I(D))=\sum\limits_{i=1}^{8}w_i-|E(D)|+1=7$ and $\mbox{pd}\,(I(D))=|E(D)|-1=7$
 by Theorem \ref{thm:3}.
\end{Example}

The following two examples show that the projective dimension and the
regularity  of the edge ideals of vertex-weighted oriented unicyclic graphs are related to direction selection in  Theorem \ref{thm:3}.
\begin{Example}  \label{exm:6}
Let $I(D)=(x_1x_2^{3},x_2x_3^2,x_4x_3^2,x_4x_1^4,x_1x_5^{2})$  be the edge ideal of a vertex-weighted oriented unicyclic  graph $D=(V(D),E(D),w)$  with $w_1=4$, $w_2=3$, $w_3=w_5=2$ and $w_4=1$. By using CoCoA, we get $\mbox{reg}\,(I(D))=9$ and $\mbox{pd}\,(I(D))=3$.  But  we have
$\mbox{reg}\,(I(D))=\sum\limits_{i=1}^{5}w_i-|E(D)|+1=12-5+1=8$ and $\mbox{pd}\,(I(D))=|E(D)|-1=5-1=4$
 by Theorem \ref{thm:3}.
\end{Example}

\begin{Example}  \label{example7}
Let $I(D)=(x_1x_2^{2},x_2x_3^2,x_3x_4^4,x_1x_4^4,x_2x_5^2,x_5x_6^2,x_3x_7^2,x_7x_8^2)$ be the edge ideal of a vertex-weighted oriented unicyclic  graph $D=(V(D),E(D),w)$  with $w_1=1$, $w_2=w_3=w_5=w_{6}=w_7=w_8=2$   and $w_4=4$. By using CoCoA, we get $\mbox{reg}\,(I(D))=11$ and $\mbox{pd}\,(I(D))=6$.  But  we have
$\mbox{reg}\,(I(D))=\sum\limits_{i=1}^{8}w_i-|E(D)|+1=17-8+1=10$ and $\mbox{pd}\,(I(D))=|E(D)|-1=7$
 by Theorem \ref{thm:3}.
\end{Example}

\begin{acknowledgment*}
This research is supported by the Natural Science Foundation of Jiangsu Province (No. BK20221353). We gratefully acknowledge the use of the computer algebra system CoCoA (\cite{C}) for our experiments.
\end{acknowledgment*}

\bibliography{zhu}

\end{document}